\documentclass[11pt,final]{article}
\usepackage{a4}
\usepackage{amsmath}%
\usepackage{amstext}%
\usepackage{amssymb}%
\usepackage{showkeys}%
\usepackage{epsfig}%
\usepackage{cite}
\usepackage{tikz}
\usepackage{subfig}
\usepackage{caption}
\usepackage{multirow}
\usepackage{booktabs}
\usepackage{floatrow}

\usepackage{listings}

\newtheorem{theorem}{Theorem}

\newtheorem{axiom}{Axiom}

\newtheorem{corollary}[theorem]{Corollary}

\newtheorem{definition}[axiom]{Definition}

\newtheorem{lemma}[theorem]{Lemma}

\newtheorem{rem}[theorem]{Remark}
\newenvironment{remark}{\rem\rm}{\endrem}

\newcounter{unnumber}

\newtheorem{prf}[unnumber]{Proof.}
\newenvironment{proof}{\prf\rm}{\hfill{$\blacksquare$}\endprf}
\newcommand{\R}{\mathbb{R}}%
\newcommand{\e}{\varepsilon}%
\newcommand{\ol}{\overline}%

\newcommand{\n}{{\nabla}}
\newcommand{\p}{{\partial}}
\newcommand{\ds}{\displaystyle}
\newcommand{\To}{\longrightarrow}
\def\a{\alpha}

\def\b{\beta}
\def\e{\epsilon}

\def\t{\theta}
\def\g{\gamma}

\def\s{\sigma}
\def\l{\lambda}
\def\<{\langle}
\def\>{\rangle}
\DeclareMathOperator*\dom{dom}%
\DeclareMathOperator*\prox{prox}%
\DeclareMathOperator*\argmin{argmin}

\DeclareMathOperator*\crit{crit}
\DeclareMathOperator*\dist{dist}

\DeclareMathOperator*\proj{proj}
\title{Approaching nonsmooth nonconvex minimization through second order proximal-gradient dynamical systems}

\author{Radu Ioan Bo\c{t} \thanks{University of Vienna, Faculty of Mathematics, Oskar-Morgenstern-Platz 1, A-1090 Vienna, Austria,
email: radu.bot@univie.ac.at. Research partially supported by FWF (Austrian Science Fund), project I 2419-N32.} \and
Ern\"{o} Robert Csetnek \thanks {University of Vienna, Faculty of Mathematics, Oskar-Morgenstern-Platz 1, A-1090 Vienna, Austria,
email: ernoe.robert.csetnek@univie.ac.at. Research partially supported by FWF (Austrian Science Fund), project P 29809-N32 and by 
an Advanced Fellowship STAR-UBB of Babe\c{s}-Bolyai University, Cluj Napoca.}
 \and Szil\'{a}rd Csaba L\'{a}szl\'{o} \thanks{Technical University of Cluj-Napoca, Department of Mathematics, Memorandumului 28, Cluj-Napoca, 
 Romania, e-mail: szilard.laszlo@math.utcluj.ro. This work was supported by a grant of Ministry of Research and Innovation, 
 CNCS - UEFISCDI, project number PN-III-P4-ID-PCE-2016-0190, within PNCDI III.}}

\begin{document}
\maketitle

\noindent \textbf{Abstract.} We investigate the asymptotic properties of the trajectories generated by a second-order dynamical system of proximal-gradient type stated in connection  with the minimization of the sum of a nonsmooth convex  and a (possibly nonconvex) smooth function. 
The convergence of the generated trajectory to a critical point of the objective is ensured provided a regularization of the objective function satisfies the Kurdyka-\L{}ojasiewicz property. We also provide convergence rates for the trajectory formulated in terms of the \L{}ojasiewicz exponent. \vspace{1ex}

\vspace{1ex}


\section{Introduction}\label{sec-intr}

Let $f:\R^n\to \R \cup \{+ \infty\}$ be a proper, convex and lower semicontinuous function  and let $g:\R^n\to \R$ be a (possibly nonconvex)
Fr\'{e}chet differentiable  function with $\b$-Lipschitz continuous gradient, i.e. 
there exists $\b\ge 0$ such that $\|\n g(x)-\n g(y)\|\le \b\|x-y\|$ for all $x,y\in\R^n.$ In this paper we investigate the optimization problem
\begin{equation}\label{opt-pb} \inf_{x\in\R^n}[f(x)+g(x)]\end{equation}
by  associating to it the following second order dynamical system of implicit-type
\begin{equation}\label{dysy}
\left\{
\begin{array}{ll}
\ddot{x}(t)+\g\dot{x}(t)+x(t)=\prox_{\l f}\big(x(t)-\l \nabla g(x(t))\big)\\
x(0)=u_0,\,\dot{x}(0)=v_0,
\end{array}
\right.
\end{equation}
where $u_0,v_0\in \R^n$, $\g,\l\in (0,+\infty)$ and 
\begin{equation}\label{intr-prox-def}  \prox\nolimits_{\lambda f} : \R^n \rightarrow \R^n, \quad \prox\nolimits_{\lambda f}(x)=\argmin_{y\in \R^n}\left \{f(y)+\frac{1}{2\lambda}\|y-x\|^2\right\},
\end{equation}
denotes the proximal point operator of $\lambda f$.

Dynamical systems of proximal-gradient type associated to optimization problems have been intensively treated in the literature. In \cite{bolte-2003}, Bolte 
studied the convergence of the trajectories of the first order dynamical system
\begin{equation}\label{intr-syst-bolte}\left\{
\begin{array}{ll}
\dot x(t)+x(t)=\proj_C\big(x(t)-\l\nabla g(x(t))\big)\\
x(0)=x_0,
\end{array}\right.\end{equation}
where $g:\R^n\rightarrow\R$ is a convex smooth function, $C\subseteq \R^n$ is a nonempty, closed and convex set, 
$x_0\in \R^n$, and $\proj_C$ denotes the projection operator on the set $C$.  The 
trajectory of \eqref{intr-syst-bolte} has been proved to converge to a minimizer of the optimization problem  
\begin{equation}\label{intr-opt-bolte} 
\inf_{x\in C} g(x),                                                             
\end{equation}
provided the latter is solvable.  We refer also to the work of Antipin \cite{antipin} for further results related to
\eqref{intr-syst-bolte}. 

The following extension of the dynamical system \eqref{intr-syst-bolte} 
\begin{equation}\label{intr-syst-abb-att}\left\{
\begin{array}{ll}
\dot x(t)+x(t)=\prox_{\l f}\big(x(t)-\l\nabla g(x(t))\big)\\
x(0)=x_0,
\end{array}\right.\end{equation}
where $f:\R^n\rightarrow\R\cup\{+\infty\}$ is a proper, convex and lower semicontinuous function, 
$g:\R^n\rightarrow\R$ is a convex smooth function and $x_0\in \R^n$, has been recently considered by Abbas and Attouch \cite{abbas-att-arx14} in relation to the optimization problem \eqref{opt-pb}. 
In case \eqref{opt-pb} is solvable, the trajectory generated by \eqref{intr-syst-abb-att} has been proved to converge to a global minimizer of it.

In connection with the optimization problem \eqref{intr-opt-bolte}, the second order  projected-gradient system  
\begin{equation}\label{intr-gr-pr}\left\{
\begin{array}{ll}
\ddot x(t) + \gamma\dot x(t) + x(t) = \proj_C(x(t)-\l\nabla g (x(t)))\\
x(0)=u_0, \dot x(0)=v_0,
\end{array}\right.\end{equation}
with damping parameter $\gamma > 0$ and step size $\l > 0$, has been considered in \cite{att-alv, antipin}.
The system \eqref{intr-gr-pr} becomes in case $C=\R^n$ 
the  so-called "heavy ball method with friction''. This nonlinear oscillator with damping 
is, in case $n=2$, a simplified version of the differential
system describing the motion of a heavy ball that rolls over the graph of $g$ and 
keeps rolling under its own inertia until friction stops it at a critical point of $g$ (see \cite{att-g-r}). 

Implicit dynamical systems related to both optimization 
problems and monotone inclusions have been  considered in the literature also by Attouch and Svaiter in \cite{att-sv2011}, Attouch, Abbas and Svaiter in \cite{abbas-att-sv} and  Attouch, Alvarez and Svaiter in \cite{att-alv-sv}. These investigations have been 
continued and extended in \cite{b-c-dyn-KM, b-c-dyn-pen-cont, b-c-dyn-sec-ord, b-c-conv-rate-cont}.

The aim of this manuscript is to study the asymptotic properties of the trajectory generated by the second order dynamical system \eqref{dysy} under convexity assumptions for $f$ and by allowing
$g$ to be nonconvex. In the same setting, a first order dynamical system of type \eqref{intr-syst-abb-att} attached to \eqref{opt-pb} has been recently studied in \cite{bc-forder-kl}.  The 
main results of the current work are Theorem \ref{convergence}, where we prove convergence of the trajectories to a critical point of the 
objective function of \eqref{opt-pb}, provided a regularization of it satisfies the Kurdyka-\L{}ojasiewicz property, and Theorem \ref{th-conv-rate}, where convergences rates by means of the \L{}ojasiewicz exponent are provided for both the trajectory and the velocity. The convergence analysis relies on methods and techniques of real algebraic geometry introduced by \L{}ojasiewicz \cite{lojasiewicz1963} and Kurdyka \cite{kurdyka1998} and extended to the nonsmooth setting by Attouch, Bolte and Svaiter \cite{att-b-sv2013} and Bolte, Sabach and Teboulle \cite{b-sab-teb}.

The explicit discretization of \eqref{dysy} with respect to the time variable $t$, with step size $h_k >0$, damping variable $\gamma_k >0$ and initial points $x_0:=u_0$ and $x_1:=v_0$ yields the iterative scheme
$$\frac{x_{k+1} - 2x_k + x_{k-1}}{h_k^2} + \gamma_k \frac{x_{k+1} - x_k}{h_k} + x_k = \prox\nolimits_{\l f}\big(x_k -\l \nabla g(x_k)\big)  \ \forall k \geq 1.$$
For $h_k=1$ this becomes
$$x_{k+1} = \left (1-\frac{1}{1+\gamma_k} \right) x_k + \frac{1}{1 + \gamma_k} \prox\nolimits_{\l f}\big(x_k -\l \nabla g(x_k)\big) + \frac{1}{1 + \gamma_k} (x_k - x_{k-1}) \ \forall k \geq 1,$$
which is a relaxed proximal-gradient algorithm for minimizing $f+g$ with inertial effects. For inertial-type algorithms we refer the reader to \cite{alvarez2000, alvarez2004, alvarez-attouch2001}. The dynamical system investigated in this paper can be seen as a continuous counterpart of the inertial-type 
algorithms presented in \cite{bcl} and \cite{ipiano}.

\section{Preliminaries}\label{sec2}

In this section we introduce some basic notions and present preliminary results that will be used in the sequel. The finite-dimensional spaces considered in the  manuscript are endowed with the Euclidean norm topology. The {\it domain} of the function  $f:\R^n\rightarrow \R \cup \{+\infty\}$ is defined by $\dom f=\{x\in\R^n:f(x)<+\infty\}$. 
We say that $f$ is {\it proper}, if $\dom f\neq\emptyset$. For the following generalized subdifferential notions and their basic properties we refer to \cite{boris-carte, rock-wets}.
Let $f:\R^n\rightarrow \R \cup \{+\infty\}$ be a proper and lower semicontinuous function. For $x\in\dom f$, the
{\it Fr\'{e}chet (viscosity) subdifferential} of $f$ at $x$ is defined as
$$\hat{\partial}f(x)= \left \{v\in\R^n: \liminf_{y\rightarrow x}\frac{f(y)-f(x)-\<v,y-x\>}{\|y-x\|}\geq 0 \right \}.$$ For
$x\notin\dom f$, we set $\hat{\partial}f(x):=\emptyset$. The {\it limiting (Mordukhovich) subdifferential} is defined at $x\in \dom f$ by
$$\partial f(x)=\{v\in\R^n:\exists x_k\rightarrow x,f(x_k)\rightarrow f(x)\mbox{ and }\exists v_k\in\hat{\partial}f(x_k),v_k\rightarrow v \mbox{ as }k\rightarrow+\infty\},$$
while for $x \notin \dom f$, we set $\partial f(x) :=\emptyset$. Notice the inclusion $\hat\partial f(x)\subseteq\partial f(x)$ for each $x\in\R^n$.

In case $f$ is convex, these notions coincide with the {\it convex subdifferential}, which means that
$\hat\partial f(x)=\partial f(x)=\{v\in\R^n:f(y)\geq f(x)+\<v,y-x\> \ \forall y\in \R^n\}$ for all $x\in\dom f$.

 We will use the following closedness criterion
concerning the graph of the limiting subdifferential: if $(x_k)_{k\geq 0}$ and $(v_k)_{k\geq 0}$ are sequences in $\R^n$ such that
$v_k\in\partial f(x_k)$ for all $k \geq 0$, $(x_k,v_k)\rightarrow (x,v)$ and $f(x_k)\rightarrow f(x)$ as $k\rightarrow+\infty$, then
$v\in\partial f(x)$.

The Fermat rule reads in this nonsmooth setting as: if $x\in\R^n$ is a local minimizer of $f$, then $0\in\partial f(x)$. Notice that
in case $f$ is continuously differentiable around $x \in \R^n$ we have $\partial f(x)=\{\nabla f(x)\}$. We denote by
$$\crit(f)=\{x\in\R^n: 0\in\partial f(x)\}$$ the set of {\it (limiting)-critical points} of $f$. We also mention the 
following subdifferential rule: if $f:\R^n\rightarrow \R \cup \{+\infty\}$ is proper and lower semicontinuous and 
$h:\R^n\rightarrow \R$ is a continuously differentiable function, then $\partial (f+h)(x)=\partial f(x)+\nabla h(x)$ for all $x\in\R^n$.

\begin{definition}\label{abs-cont} \rm (see, for instance, \cite{att-sv2011, abbas-att-sv}) A function $x:[0,+\infty)\rightarrow \R^n$  is said to be locally absolutely continuous, if is absolutely continuous on every interval $[0,T],\,T>0$, that is, one of the
following equivalent properties holds:

(i)  there exists an integrable function $y:[0,T]\rightarrow \R^n$ such that $$x(t)=x(0)+\int_0^t y(s)ds \ \ \forall t\in[0,T];$$

(ii) $x$ is continuous and its distributional derivative is Lebesgue integrable on $[0,T]$;

(iii) for every $\varepsilon > 0$, there exists $\eta >0$ such that for any finite family of intervals $I_k=(a_k,b_k)\subseteq [0,T]$ we have the implication
$$\left(I_k\cap I_j=\emptyset \mbox{ and }\sum_k|b_k-a_k| < \eta\right)\Longrightarrow \sum_k\|x(b_k)-x(a_k)\| < \varepsilon.$$
\end{definition}

\begin{remark}\label{rem-abs-cont}\rm (a) It follows from the definition that an absolutely continuous function is differentiable almost
everywhere, its derivative coincides with its distributional derivative almost everywhere and one can recover the function from its derivative $\dot x=y$ by the integration formula (i).

(b) If $x:[0,T]\rightarrow \R^n$ (where $T>0$) is absolutely continuous and $B:\R^n\rightarrow \R^n$ is $L$-Lipschitz continuous
(where $L\geq 0$), then the function $z=B\circ x$ is absolutely continuous, too. This can be easily seen by using the characterization of absolute continuity in
Definition \ref{abs-cont}(iii). Moreover, $z$ is almost everywhere differentiable and the inequality $\|\dot z (\cdot)\|\leq L\|\dot x(\cdot)\|$ holds almost everywhere.
\end{remark}

Further, we recall the following result of Br\'ezis \cite{B}.

\begin{lemma}\label{deriv} Let $f:\R^n\To\R\cup\{+\infty\}$ be a proper, convex and lower semicontinuous function. Let $x\in L^2([0,T],\R^n),\, T>0,$ be absolutely continuous such that $\dot{x}\in L^2([0,T],\R^n)$ and $x(t)\in\dom f$ for almost every $t\in[0,T].$ Assume that there exists $\xi\in L^2([0,T],\R^n)$ such that $\xi(t)\in\p f(x(t))$ for almost every $t\in[0,T].$ Then the function $t\To f(x(t))$ is absolutely continuous and for every $t$ such that $x(t)\in\dom \p f$ we have
$$\frac{d}{dt}f(x(t))=\<\dot{x}(t),h\>,\,\forall h\in\p f(x(t)).$$
\end{lemma}

The following central results will be used when proving the convergence of the trajectories generated by the dynamical system \eqref{dysy}; see, for example, \cite[Lemma 5.1]{abbas-att-sv} and \cite[Lemma 5.2]{abbas-att-sv}, respectively.

\begin{lemma}\label{fejer-cont1} Suppose that $F:[0,+\infty)\rightarrow\R$ is locally absolutely continuous and bounded below and that
there exists $G\in L^1([0,+\infty))$ such that for almost every $t \in [0,+\infty)$ $$\frac{d}{dt}F(t)\leq G(t).$$
Then there exists $\lim_{t\rightarrow +\infty} F(t)\in\R$.
\end{lemma}

\begin{lemma}\label{fejer-cont2}  If $1 \leq p < \infty$, $1 \leq r \leq \infty$, $F:[0,+\infty)\rightarrow[0,+\infty)$ is
locally absolutely continuous, $F\in L^p([0,+\infty))$, $G:[0,+\infty)\rightarrow\R$, $G\in  L^r([0,+\infty))$ and
for almost every $t \in [0,+\infty)$ $$\frac{d}{dt}F(t)\leq G(t),$$ then $\lim_{t\rightarrow +\infty} F(t)=0$.
\end{lemma}

\section{Existence and uniqueness of the trajectories}\label{sec3}

Existence and uniqueness of the trajectories of \eqref{dysy} are obtained in the framework of the global version of the Cauchy-Lipschitz Theorem 
(see for instance \cite[Theorem 17.1.2(b)]{ABG}), by rewriting \eqref{dysy} as a first order dynamical system in a suitable product space
and by employing the Lipschitz continuity of the proximal operator and of  the gradient.

\begin{theorem}\label{uniq} For every starting points $u_0,v_0\in \R^n$, the dynamical system \eqref{dysy} has a unique global solution $x\in C^2([0,+\infty),\R^n)$.
\end{theorem}
\begin{proof}

By making use of the notation $X(t)=(x(t),\dot{x}(t))$, the system \eqref{dysy} can be rewritten as
\begin{equation}\label{dysy1}
\left\{
\begin{array}{ll}
\dot{X}(t)=F(X(t))\\
X(0)=(u_0,v_0),
\end{array}
\right.
\end{equation}
where $F:\R^n\times\R^n\To \R^n\times\R^n,\, F(u,v)=\left(v, \prox_{\l f}\big(u-\l \nabla g(u)\big)-\g v-u\right).$

We prove the existence and uniqueness of a global solution of \eqref{dysy1} by using the Cauchy-Lipschitz Theorem. To this aim it is enough to show that $F$ is globally Lipschitz continuous.
Let be $(u,v), (\ol{u}, \ol{v}) \in \R^n \times \R^n$. We have
\begin{align*}
\|F(u,v)-F(\ol{u},\ol{v})\| \!& = \!\! \left\|\left(v-\ol{v},\prox\nolimits_{\l f}\big(u-\l \nabla g(u)\big)-\prox\nolimits_{\l f}\big(\ol{u}-\l \nabla g(\ol{u})\big)+\g(\ol{v}- v)+(\ol{u}-u)\right)\right\|\\
\!& = \!\!\sqrt{\|v-\ol{v}\|^2+\!\|\prox\nolimits_{\l f}\big(u-\l \nabla g(u)\big)-\prox\nolimits_{\l f}\big(\ol{u}-\l \nabla g(\ol{u})\big)+\!\g(\ol{v}- v)+(\ol{u}-u)\|^2}.
\end{align*}
We have
\begin{align*}
& \|\prox\nolimits_{\l f}\big(u-\l \nabla g(u)\big)-\prox\nolimits_{\l f}\big(\ol{u}-\l \nabla g(\ol{u})\big)+\g(\ol{v}- v)+(\ol{u}-u)\|^2=\\
& \|\prox\nolimits_{\l f}\big(u-\l \nabla g(u)\big)-\prox\nolimits_{\l f}\big(\ol{u}-\l \nabla g(\ol{u})\big)\|^2+\g^2\|\ol{v}- v\|^2+\|\ol{u}-u\|^2+\\
& 2\g\<\prox\nolimits_{\l f}\big(u-\l \nabla g(u)\big)-\prox\nolimits_{\l f}\big(\ol{u}-\l \nabla g(\ol{u})\big),\ol{v}- v\>+\\
& 2\<\prox\nolimits_{\l f}\big(u-\l \nabla g(u)\big)-\prox\nolimits_{\l f}\big(\ol{u}-\l \nabla g(\ol{u})\big),\ol{u}-u \>+\\
& 2\g\<\ol{v}- v,\ol{u}-u\>.
\end{align*}
By the nonexpansiveness of $\prox_{\l f}$ and the $\b$-Lipschitz property of $\n g$ we have
$$\|\prox\nolimits_{\l f}\big(u-\l \nabla g(u)\big)-\prox\nolimits_{\l f}\big(\ol{u}-\l \nabla g(\ol{u})\big)\|\le\|(u-\ol{u})-\l(\nabla g(u)-\nabla g(\ol{u})\|\le(1+\l\b)\|u-\ol{u}\|.$$
On the other hand,
\begin{align*}
2\g\<\prox\nolimits_{\l f}\big(u-\l \nabla g(u)\big)-\prox\nolimits_{\l f}\big(\ol{u}-\l \nabla g(\ol{u})\big),\ol{v}- v\> & \le \\
\g\|\prox\nolimits_{\l f}\big(u-\l \nabla g(u)\big)-\prox\nolimits_{\l f}\big(\ol{u}-\l \nabla g(\ol{u})\big)\|^2+\g\|\ol{v}- v\|^2 & \le\\
\g(1+\l\b)^2\|u-\ol{u}\|^2+\g\|\ol{v}- v\|^2,&
\end{align*}
\begin{align*}
2\<\prox\nolimits_{\l f}\big(u-\l \nabla g(u)\big)-\prox\nolimits_{\l f}\big(\ol{u}-\l \nabla g(\ol{u})\big),\ol{u}-u \> & \le\\
\|\prox\nolimits_{\l f}\big(u-\l \nabla g(u)\big)-\prox\nolimits_{\l f}\big(\ol{u}-\l \nabla g(\ol{u})\big)\|^2+\|\ol{u}-u\|^2 & \le\\
(1+(1+\l\b)^2)\|u-\ol{u}\|^2&
\end{align*}
and
$$2\g\<\ol{v}- v,\ol{u}-u\>\le\g\|\ol{v}- v\|^2+\g\|u-\ol{u}\|^2.$$
Consequently,
\begin{align*}
\|\prox\nolimits_{\l f}\big(u-\l \nabla g(u)\big)-\prox\nolimits_{\l f}\big(\ol{u}-\l \nabla g(\ol{u})\big)+\g(\ol{v}- v)+(\ol{u}-u)\|^2& \le\\
(\g+2)((1+\l\b)^2+1)\|u-\ol{u}\|^2+(\g^2+2\g)\|\ol{v}- v\|^2,&
\end{align*}
which leads to
$$\|F(u,v)-F(\ol{u},\ol{v})\|\le \sqrt{(\g+1)^2\|\ol{v}- v\|^2+(\g+2)((1+\l\b)^2+1)\|u-\ol{u}\|^2}\le L_1\|(u,v)-(\ol{u},\ol{v})\|,$$
where $L_1:=\sqrt{\max\big((\g+1)^2,(\g+2)((1+\l\b)^2+1)\big)}.$

Consequently, $F$ is globally Lipschitz continuous, which implies that \eqref{dysy1} has a global solution $X\in C^1([0,+\infty),\R^n\times\R^n)$. This shows that $x\in C^2([0,+\infty),\R^n).$
\end{proof}

\begin{remark} Another Lipschitz constant can be obtained by using the inequalities: 
\begin{align*}
2\g\<\prox\nolimits_{\l f}\big(u-\l \nabla g(u)\big)-\prox\nolimits_{\l f}\big(\ol{u}-\l \nabla g(\ol{u})\big),\ol{v}- v\> & \le \\
2\g\|\prox\nolimits_{\l f}\big(u-\l \nabla g(u)\big)-\prox\nolimits_{\l f}\big(\ol{u}-\l \nabla g(\ol{u})\big)\|\|\ol{v}- v\| & \le 2\g(1+\l\b)\|u-\ol{u}\|\|\ol{v}- v\|,
\end{align*}
\begin{align*}
2\<\prox\nolimits_{\l f}\big(u-\l \nabla g(u)\big)-\prox\nolimits_{\l f}\big(\ol{u}-\l \nabla g(\ol{u})\big),\ol{u}-u \> & \le\\
2\|\prox\nolimits_{\l f}\big(u-\l \nabla g(u)\big)-\prox\nolimits_{\l f}\big(\ol{u}-\l \nabla g(\ol{u})\big)\|\|\ol{u}- u\| & \le
2(1+\l\b)\|\ol{u}- u\|^2,&
\end{align*}
$$2\g\<\ol{v}- v,\ol{u}-u\>\le2\g\|\ol{u}-u\|\|\ol{v}- v\|,$$
and
$$2\|\ol{u}-u\|\|\ol{v}- v\|\le\|\ol{u}-u\|^2+\|\ol{v}- v\|^2.$$

In this case one obtains the Lipschitz constant
$$L_2:=\sqrt{\max((\g+1)^2+\g\l\b,(2+\l\b)^2+\g(2+\l\b))}.$$
\end{remark}

\begin{remark}
Considering again the setting of the proof of Theorem \ref{uniq}, from Remark \ref{rem-abs-cont}(b) it follows that $\ddot{X}$ exists almost everywhere on $[0,+\infty)$ and that for almost every $t\in[0,+\infty)$ one has
$$\|\ddot{X}(t)\|\le L_1\|\dot{X}(t)\|= \sqrt{\max\big((\g+1)^2,(\g+2)((1+\l\b)^2+1)\big)}\|\dot{X}(t)\|.$$
Hence, $\sqrt{\|\ddot{x}(t)\|^2+\|x^{(3)}(t)\|^2}\le \sqrt{\max\big((\g+1)^2,(\g+2)((1+\l\b)^2+1)\big)}\sqrt{\|\dot{x}(t)\|^2+\|\ddot{x}(t)\|^2},$ for almost every $t\in[0,+\infty)$, or, equivalently, 
\begin{align}\label{third-deriv}
\|x^{(3)}(t)\|^2\le & \max\big((\g+1)^2,(\g+2)((1+\l\b)^2+1)\big)\|\dot{x}(t)\|^2+ \nonumber\\
& (\max\big((\g+1)^2,(\g+2)((1+\l\b)^2+1)\big)-1)\|\ddot{x}(t)\|^2.
\end{align}

Similarly, by  using $L_2$, one obtains for almost every $t\in[0,+\infty)$
\begin{align}\label{third-deriv1}
\|x^{(3)}(t)\|^2\le & \max((\g+1)^2+\g\l\b,(2+\l\b)^2+\g(2+\l\b))\|\dot{x}(t)\|^2+ \nonumber \\
& (\max((\g+1)^2+\g\l\b,(2+\l\b)^2+\g(2+\l\b))-1)\|\ddot{x}(t)\|^2.
\end{align}

\end{remark}
\begin{remark} Obviously, $L_1>2$ and $L_2>2.$ One can easily verify that $L_2\le L_1$, provided $\g\le\sqrt{3}.$ Moreover, if $\g\le\sqrt{3},$ then
 $$L_2=\sqrt{(2+\l\b)^2+\g(2+\l\b)}.$$
 However, for $\g>\sqrt{3}$, one may have $L_2>L_1$ and also $L_2<L_1.$ Indeed, for $\g=2$ and $\l\b=\frac{1}{10},$ it holds
 $$L_2=\sqrt{9,2}>3  =L_1,$$ 
while for $\g=2$ and $\l\b=1$ it holds
 $$L_2=\sqrt{15}<\sqrt{20}=L_1.$$
\end{remark}

\section{Asymptotic analysis}\label{sec4}

In this section we will address the asymptotic behaviour of the trajectory generated by the second order dynamical system \eqref{dysy}. We begin the analysis with some technical results. 

\begin{lemma}\label{l8} Suppose that $f+g$ is bounded from bellow and $\g,\l>0$ satisfy the following set of conditions:
 $$(\rho)\quad \left\{\begin{array}{lll}
 A=\ds-\frac12\frac\g\l+\frac{\b}{2}(L^2+2\g^2+1)<0\\
 B=\ds-\frac{1}{2L^2}\frac\g\l+\frac{\b}{2}(L^2+\g^2+1)<0\\
 C=\ds-\frac{(2L^2+1)}{ (L^2+1)^2}\g^2+3\b\g\l-1<0,
 \end{array} \right.$$
where  $L:=\min(L_1,L_2)$, and $L_1,L_2$ were defined as
$$L_1=\sqrt{\max\big((\g+1)^2,(\g+2)((1+\l\b)^2+1)\big)}$$ 
and 
$$L_2=\sqrt{\max((\g+1)^2+\g\l\b,(2+\l\b)^2+\g(2+\l\b))}.$$
 For $u_0,v_0\in\R^n$, let $x\in C^2([0,+\infty),\R^n)$ be the unique global solution of \eqref{dysy}. Then the following statements are true
\begin{itemize}
  \item[(a)] $\dot{x}\in L^2([0,+\infty),\R^n)$ and $\lim_{t\To+\infty}\dot{x}(t)=0$;
  \item[(b)] $\ddot{x}\in L^2([0,+\infty),\R^n)$ and $\lim_{t\To+\infty}\ddot{x}(t)=0$;
  \item[(c)] $\exists\,\lim_{t\To+\infty}(f+g)(\ddot{x}(t)+\g\dot{x}(t)+x(t))\in\R.$
\end{itemize}
\end{lemma}
\begin{proof} Let $T>0.$ Since $x\in C^2([0,T],\R^n)$, we have $x,\dot{x},\ddot{x}\in L^2([0,T],\R^n).$ Further, by the $\b$-Lipschitz property of $\n g$ we have $\n g\in L^2([0,T],\R^n).$ Moreover, \eqref{third-deriv} ensures that $x^{(3)}\in L^2([0,T],\R^n).$

According to \eqref{dysy}, we have $\ddot{x}(t)+\g\dot{x}(t)+x(t)=\prox_{\l f}\big(x(t)-\l \nabla g(x(t))\big)$ for all $t\in[0,+\infty)$, hence
\begin{equation}\label{inc}
-\frac1\l\ddot{x}(t)-\frac\g\l\dot{x}(t)-\n g(x(t))\in\p f(\ddot{x}(t)+\g\dot{x}(t)+x(t)).
\end{equation}
On the other hand, $\xi(t)=-\frac1\l\ddot{x}(t)-\frac\g\l\dot{x}(t)-\n g(x(t))\in L^2([0,T],\R^n)$, hence by Lemma \ref{deriv} we have that $t\To f(\ddot{x}(t)+\g\dot{x}(t)+x(t))$ is absolutely continuous and
\begin{equation}
\frac{d}{dt}f(\ddot{x}(t)+\g\dot{x}(t)+x(t))=\left\< x^{(3)}(t)+\g\ddot{x}(t)+\dot{x}(t),-\frac1\l\ddot{x}(t)-\frac\g\l\dot{x}(t)-\n g(x(t))\right\>
\end{equation}
for almost every $t\in[0,T].$

Obviously,
\begin{equation}
\frac{d}{dt}g(\ddot{x}(t)+\g\dot{x}(t)+x(t))=\left\< x^{(3)}(t)+\g\ddot{x}(t)+\dot{x}(t),\n g( \ddot{x}(t)+\g\dot{x}(t)+x(t))\right\>
\end{equation}
for almost every $t\in[0,T].$
By summing up the last two equalities we get
\begin{align*}
& \frac{d}{dt}(f+g)(\ddot{x}(t)+\g\dot{x}(t)+x(t))=\\
&\left\< x^{(3)}(t)+\g\ddot{x}(t)+\dot{x}(t),\n g( \ddot{x}(t)+\g\dot{x}(t)+x(t))-\n g(x(t))-\frac1\l\ddot{x}(t)-\frac\g\l\dot{x}(t)\right\>=\\
&-\frac{1}{2\l}\frac{d}{dt}\|\ddot{x}(t)\|^2-\frac{1+\g^2}{2\l}\frac{d}{dt}\|\dot{x}(t)\|^2-\frac{\g}{\l}\|\dot{x}(t)\|^2-\frac{\g}{\l}\|\ddot{x}(t)\|^2
-\frac{\g}{\l}\<x^{(3)}(t),\dot{x}(t)\>+\\
&\left\< x^{(3)}(t)+\g\ddot{x}(t)+\dot{x}(t),\n g( \ddot{x}(t)+\g\dot{x}(t)+x(t))-\n g(x(t))\right\>
\end{align*}
for almost every $t\in[0,T].$ It is easy to check that $\<x^{(3)}(t),\dot{x}(t)\>=\frac12\cdot\frac{d^2}{dt^2}\|\dot{x}(t)\|^2-\|\ddot{x}(t)\|^2$ for almost every $t\in[0,+\infty)$. Let $c\in(0,1).$ We have $$-\frac{\g}{\l}\<x^{(3)}(t),\dot{x}(t)\>=-c\frac{\g}{\l}\<x^{(3)}(t),\dot{x}(t)\>-(1-c)\frac{\g}{\l}\<x^{(3)}(t),\dot{x}(t)\>$$ and
$$-(1-c)\frac{\g}{\l}\<x^{(3)}(t),\dot{x}(t)\>\le \left(a\|x^{(3)}(t)\|^2+b\|\dot{x}(t)\|^2\right),$$ where $ab=\frac{\g^2(1-c)^2}{4\l^2},$ hence by using \eqref{third-deriv} and \eqref{third-deriv1} one obtains that for almost every $t\in[0,+\infty)$ 
$$-(1-c)\frac{\g}{\l}\<x^{(3)}(t),\dot{x}(t)\>\le (a L^2+b) \|\dot{x}(t)\|^2+a(L^2-1)\|\ddot{x}(t)\|^2.$$

Consequently,  for almost every $t\in[0,+\infty)$ we have
\begin{equation}\label{prodcoef}
-\frac{\g}{\l}\<x^{(3)}(t),\dot{x}(t)\>\le -\frac{c\g}{2\l}\cdot\frac{d^2}{dt^2}\|\dot{x}(t)\|^2+(a L^2+b)\|\dot{x}(t)\|^2+\left(aL^2+\frac{c\g}{\l}-a\right)\|\ddot{x}(t)\|^2.
\end{equation}

Further, for almost every $t\in[0,+\infty)$ we have
\begin{align*}
\left\< x^{(3)}(t)+\g\ddot{x}(t)+\dot{x}(t),\n g( \ddot{x}(t)+\g\dot{x}(t)+x(t))-\n g(x(t))\right\> & \le\\
\b\|\ddot{x}(t)+\g\dot{x}(t)\|\| x^{(3)}(t)+\g\ddot{x}(t)+\dot{x}(t)\| & \le \\
\b(\|\ddot{x}(t)+\g\dot{x}(t)\|\|x^{(3)}(t)\|+\|\ddot{x}(t)+\g\dot{x}(t)\|\|\g\ddot{x}(t)+\dot{x}(t)\|) & \le\\
\frac{\b}{2}(2\| \ddot{x}(t)+\g\dot{x}(t)\|^2+\|x^{(3)}(t)\|^2+\|\g\ddot{x}(t)+\dot{x}(t)\|^2) & =\\
\frac{\b}{2}\big((2+\g^2)\|\ddot{x}(t)\|^2+(2\g^2+1)\|\dot{x}(t)\|^2+\|x^{(3)}(t)\|^2+6\g\<\ddot{x}(t),\dot{x}(t)\>\big) & =\\
\frac{\b}{2}\left((2+\g^2)\|\ddot{x}(t)\|^2+(2\g^2+1)\|\dot{x}(t)\|^2+\|x^{(3)}(t)\|^2+3\g\frac{d}{dt}\|\dot{x}(t)\|^2\right).&
\end{align*}

By using \eqref{third-deriv} and  \eqref{third-deriv1} one obtains for almost every $t\in[0,T]$ 
$$\|x^{(3)}(t)\|^2\le L^2\|\dot{x}(t)\|^2+(L^2-1)\|\ddot{x}(t)\|^2,$$
hence
\begin{align*}
\left\< x^{(3)}(t)+\g\ddot{x}(t)+\dot{x}(t),\n g( \ddot{x}(t)+\g\dot{x}(t)+x(t))-\n g(x(t))\right\> & \le\\
\frac{\b}{2}( L^2+\g^2+1)\|\ddot{x}(t)\|^2+\frac{\b}{2}( L^2+2\g^2+1)\|\dot{x}(t)\|^2+3\frac{\b}{2}\g\frac{d}{dt}\|\dot{x}(t)\|^2.&
\end{align*}

Consequently, for almost every $t\in[0,T]$ we have
\begin{align*}
\frac{d}{dt}(f+g)( \ddot{x}(t)+\g\dot{x}(t)+x(t))+\frac{1}{2\l}\frac{d}{dt}\|\ddot{x}(t)\|^2+\frac{(1+\g^2)-3\l\b\g}{2\l}\frac{d}{dt}\|\dot{x}(t)\|^2+
\frac{c\g}{2\l}\cdot\frac{d^2}{dt^2}\|\dot{x}(t)\|^2 & \le \\
\left((c-1)\frac{\g}{\l}+aL^2-a+\frac{\b}{2}(L^2+\g^2+1)\right)\|\ddot{x}(t)\|^2+\left(-\frac\g\l+a L^2+b+\frac{\b}{2}( L^2+2\g^2+1)\right)\|\dot{x}(t)\|^2.&
\end{align*}

Recall that $a,b$ and $c$ have been arbitrarily chosen such that $c\in(0,1)$  and $ab=\frac{\g^2(1-c)^2}{4\l^2}.$ 

We chose 
$$c:=\frac{L^2}{L^2+1},\,a:=\frac{\g}{2(L^2+1)L^2\l} \ \mbox{and} \ b:=\frac{L^2\g}{2(L^2+1)\l}.$$

Then, for almost every $t\in[0,T]$ we have
\begin{align}\label{ederiv}
\frac{d}{dt}\left[(f+g)( \ddot{x}(t)+\g\dot{x}(t)+x(t))+\frac{1}{2\l}\|\ddot{x}(t)\|^2+\frac{c^2\g^2-C}{2\l}\|\dot{x}(t)\|^2+ \frac{2 c\g}{2\l}\<\ddot{x}(t),\dot{x}(t)\>\right] & \le \nonumber\\
A\|\dot{x}(t)\|^2+B\|\ddot{x}(t)\|^2. &
\end{align}

By integration we get
\begin{align*}
(f+g)( \ddot{x}(T)+\g\dot{x}(T)+x(T))+\frac{1}{2\l}\|\ddot{x}(T)\|^2+\frac{c^2\g^2-C}{2\l}\|\dot{x}(T)\|^2+
\frac{2 c\g}{2\l}\<\ddot{x}(T),\dot{x}(T)\> & \le\\
(f+g)(\ddot{x}(0)+\g\dot{x}(0)+x(0))+\frac{1}{2\l}\|\ddot{x}(0)\|^2+\frac{c^2\g^2-C}{2\l}\|\dot{x}(0)\|^2+
\frac{2 c\g}{2\l}\<\ddot{x}(0),\dot{x}(0)\>+ & \\
A\int_0^T \|\dot{x}(t)\|^2 dt+B\int_0^T \|\ddot{x}(t)\|^2 dt.&
\end{align*}

In other words,
\begin{align}\label{integ}
(f+g)( \ddot{x}(T)+\g\dot{x}(T)+x(T))+\frac{1}{2\l}\|\ddot{x}(T)+c\g\dot{x}(T)\|^2-\frac{C}{2\l}\|\dot{x}(T)\|^2 & \le \nonumber\\
(f+g)( \ddot{x}(0)+\g\dot{x}(0)+x(0))+\frac{1}{2\l}\|\ddot{x}(0)+c\g\dot{x}(0)\|^2-\frac{C}{2\l}\|\dot{x}(0)\|^2+ & \nonumber \\
A\int_0^T \|\dot{x}(t)\|^2 dt+B\int_0^T \|\ddot{x}(t)\|^2 dt.& 
\end{align}
By using that $A<0,\,B<0,\,C<0$ and $f+g$ is bounded from below, and by taking into account that $T>0$ has been arbitrary chosen, we obtain that
$\dot{x},\,\ddot{x}\in L^2([0,+\infty),\R^n).$ Moreover, from \eqref{third-deriv} we obtain that $x^{(3)}\in L^2([0,+\infty),\R^n).$

Now, by using Lemma \ref{fejer-cont2} and the fact that for almost every $t\in[0,+\infty)$ we have
$$\frac{d}{dt}\|\dot{x}(t)\|^2=2\<\dot{x}(t),\ddot{x}(t)\>\le\|\dot{x}(t)\|^2+\|\ddot{x}(t)\|^2$$
and
$$\frac{d}{dt}\|\ddot{x}(t)\|^2=2\<\ddot{x}(t),{x}^{(3)}(t)\>\le\|\ddot{x}(t)\|^2+\|{x}^{(3)}(t)\|^2,$$
we obtain that $\lim_{t\To+\infty}\dot{x}(t)=0$ and $\lim_{t\To+\infty}\ddot{x}(t)=0.$

Since $T>0$ has been arbitrary chosen, we get from \eqref{ederiv}  that for almost every $t\in[0,+\infty)$ 
$$\frac{d}{dt}\left[(f+g)( \ddot{x}(t)+\g\dot{x}(t)+x(t))+\frac{1}{2\l}\|\ddot{x}(t)\|^2+\frac{c^2\g^2-C}{2\l}\|\dot{x}(t)\|^2+
\frac{c\g}{\l}\<\ddot{x}(t),\dot{x}(t)\>\right]\le 0.$$
Now using Lemma \ref{fejer-cont1} we obtain that the limit
$$\lim_{t\To+\infty}\left[(f+g)( \ddot{x}(t)+\g\dot{x}(t)+x(t))+\frac{1}{2\l}\|\ddot{x}(t)\|^2+\frac{c^2\g^2-C}{2\l}\|\dot{x}(t)\|^2+
\frac{ c\g}{\l}\<\ddot{x}(t),\dot{x}(t)\>\right]$$
 exists and is finite. Since
$$\lim_{t\To+\infty}\left[\frac{1}{2\l}\|\ddot{x}(t)\|^2+\frac{c^2\g^2-C}{2\l}\|\dot{x}(t)\|^2+
\frac{c\g}{\l}\<\ddot{x}(t),\dot{x}(t)\>\right]=0,$$
one obtains that
$$\lim_{t\To+\infty}(f+g)( \ddot{x}(t)+\g\dot{x}(t)+x(t))\in\R.$$
\end{proof}

\begin{remark} The choice $\g\l\b\le \frac{1}{3}$ guarantees that $C<0.$ Moreover, in this case $B> A.$ Indeed,
$$B-A=\frac{\g}{2\l}\left(1-\frac{1}{L^2}-\g\l\b\right)\ge \frac{\g}{2\l}\left(\frac23-\frac{1}{L^2}\right)> 0.$$
\end{remark}

\begin{corollary} Suppose that $f+g$ is bounded from bellow and $\sqrt{3}\ge\g>0,\l>0$ satisfy the following  condition
$$\ds -\frac{1}{(2+\l\b)^2+\g(2+\l\b)}\frac\g\l+\b((2+\l\b)^2+\g(2+\l\b)+\g^2+1)<0.$$
 For $u_0,v_0\in\R^n$, let $x\in C^2([0,+\infty),\R^n)$ be the unique global solution of \eqref{dysy}. Then the following statements are true
\begin{itemize}
  \item[(a)] $\dot{x}\in L^2([0,+\infty),\R^n)$ and $\lim_{t\To+\infty}\dot{x}(t)=0$;
  \item[(b)] $\ddot{x}\in L^2([0,+\infty),\R^n)$ and $\lim_{t\To+\infty}\ddot{x}(t)=0$;
  \item[(c)] $\exists\,\lim_{t\To+\infty}(f+g)(\ddot{x}(t)+\g\dot{x}(t)+x(t))\in\R.$
\end{itemize}
\end{corollary}
\begin{proof} The condition $\g\le\sqrt{3}$ ensures that $L=\sqrt{(2+\l\b)^2+\g(2+\l\b)},$ hence $$2B=-\frac{1}{(2+\l\b)^2+\g(2+\l\b)}\frac\g\l+\b((2+\l\b)^2+\g(2+\l\b)+\g^2+1)<0.$$ 
Under these auspicies, it can proved that  $\g\l\b\le\frac13,$ hence, according to the previous remark, $C<0$ and $A<0$. The statement follows from  Lemma \ref{l8}.
\end{proof}

\begin{lemma}\label{l9} Assume that $f+g$ is bounded from below and $\g,\l$ satisfy the set of conditions $(\rho).$ For $u_0,v_0\in\R^n$, let $x\in C^2([0,+\infty),\R^n)$ be the unique global solution of \eqref{dysy}. Then the set of limit points of $x$, which we denote by  $\omega(x)$, is a subset of the set of critical points of $f+g.$ In other words,
$$\omega(x):=\{\ol{x}\in\R^n:\exists t_k\To\infty\mbox{ such that }x(t_k)\To\ol{x},\,k\To + \infty\}\subseteq\crit(f+g).$$
\end{lemma}

\begin{proof} Let $\ol{x}\in\omega(x)$ and $t_k\To +\infty\mbox{ such that }x(t_k)\To\ol{x},\,k\To+\infty$. 
We have to show that $0\in\p(f+g)(\ol{x}).$ From \eqref{inc} we have for every $k \geq 0$
$$-\frac1\l\ddot{x}(t_k)-\frac\g\l\dot{x}(t_k)-\n g(x(t_k))\in\p f(\ddot{x}(t_k)+\g\dot{x}(t_k)+x(t_k))$$
hence,
\begin{align*}
v_k=-\frac1\l\ddot{x}(t_k)-\frac\g\l\dot{x}(t_k)-\n g(x(t_k))+\n g(\ddot{x}(t_k)+\g\dot{x}(t_k)+x(t_k)) & \in\\\
\p f(\ddot{x}(t_k)+ \g\dot{x}(t_k)+x(t_k))+\n g(\ddot{x}(t_k)+\g\dot{x}(t_k)+x(t_k)) & =\\
\p(f+g)(\ddot{x}(t_k)+\g\dot{x}(t_k)+x(t_k)) & =\p(f+g)(u_k),
\end{align*}
where $u_k:=\ddot{x}(t_k)+\g\dot{x}(t_k)+x(t_k)$.

According to Lemma \ref{l8}, $\lim_{k\To+\infty}\dot{x}(t_k)=0$ and $\lim_{k\To+\infty}\ddot{x}(t_k)=0.$ Further, $\n g$ is continuous, hence
 $\lim_{k\To+\infty}[-\n g(x(t_k))+\n g(\ddot{x}(t_k)+\g\dot{x}(t_k)+x(t_k))]=-\n g(\ol{x})+\n g(\ol{x})=0.$ Consequently,
 $$\lim_{k\To+\infty}(u_k, v_k)=(\ol{x},0).$$
We show that $\lim_{k\To+\infty}(f+g)(u_k)=(f+g)(\ol{x}).$ Since $f$ is lower semicontinuous, one has
$$\liminf_{k\To+\infty}f(u_k)\ge f(\ol{x}).$$

Further  we have for every $k \geq 0$
\begin{align*}
u_k=\ddot{x}(t_k)+\g\dot{x}(t_k)+x(t_k)=\prox\nolimits_{\l f}\big(x(t_k)-\l \nabla g(x(t_k))\big) & =\\
\argmin_{y\in\R^n}\left[f(y)+\frac{1}{2\l}\|y-(x(t_k)-\l\n g(x(t_k))\|^2\right] & =\\
\argmin_{y\in\R^n}\left[f(y)+\frac{1}{2\l}\|y-x(t_k)\|^2+\<y-x(t_k),\n g(x(t_k))\>+\frac{\l}{2}\|\n g(x(t_k))\|^2\right] & =\\
\argmin_{y\in\R^n}\left[f(y)+\frac{1}{2\l}\|y-x(t_k)\|^2+\<y-x(t_k),\n g(x(t_k))\>\right].&
\end{align*}

Hence, for every $k \geq 0$ we have
$$f(u_k)+\frac{1}{2\l}\|u_k-x(t_k)\|^2+\<u_k-x(t_k),\n g(x(t_k))\>\le f(\ol{x})+\frac{1}{2\l}\|\ol{x}-x(t_k)\|^2+\<\ol{x}-x(t_k),\n g(x(t_k))\>.$$
Taking the limit superior as $k\To+\infty$, we obtain
$$\limsup_{k\To+\infty}f(u_k)\le f(\ol{x}).$$
This shows that $\lim_{k\To+\infty}f(u_k)= f(\ol{x})$ and, since $g$ is continuous,  we obtain
$$\lim_{k\To+\infty}(f+g)(u_k)= (f+g)(\ol{x}).$$
By the  closedness criterion of the graph of the limiting subdifferential it follows that
$$0\in\p(f+g)(\ol{x}).$$
\end{proof}

\begin{lemma}\label{l10} Assume that $f+g$ is bounded from below and $\g,\l$ satisfy the set of conditions $(\rho)$, and let the constants $L, A, B$ and $C$ be defined as in Lemma \ref{l8}. For $u_0,v_0\in\R^n$, let $x\in C^2([0,+\infty),\R^n)$  be the unique global solution of \eqref{dysy}. Consider the function
$$H:\R^n\times \R^n\times \R^n\To\R\cup\{+\infty\},\,H(u,v,w)=(f+g)(u)+\frac{1}{2\l}\|u-v\|^2-\frac{C}{2\l}\|w\|^2.$$
Then the following statements are true

(H$_1$) for almost every $t\in[0,+\infty)$ it holds
$$\frac{d}{dt}\left(H(\ddot{x}(t)+\g\dot{x}(t)+x(t),\g(1-c)\dot{x}(t)+x(t),\dot{x}(t))\right)\le0$$
and the limit $$\lim_{t\To+\infty}H(\ddot{x}(t)+\g\dot{x}(t)+x(t),\g(1-c)\dot{x}(t)+x(t),\dot{x}(t))$$
exists and is finite, where $c=\frac{L^2}{L^2+1}$;

(H$_2$) for almost every $t\in[0,+\infty)$ and for every $a\ge 0$ we have
\begin{align*}
w(t)=\left(-\n g(x(t))+\n g(\ddot{x}(t)+\g\dot{x}(t)+x(t))-\frac1\l a\g\dot{x}(t)),-\frac1\l(\ddot{x}(t)+(1-a)\g\dot{x}(t)),-\frac{C}{\l}\dot{x}(t)\right) & \in\\
\p H(\ddot{x}(t)+\g\dot{x}(t)+x(t),\g a\dot{x}(t)+x(t),\dot{x}(t)) &
\end{align*}
and
$$\|w(t)\|\le\left(\b+\frac1\l\right)\|\ddot{x}(t)\|+\frac{\b\l\g+(2a+1)\g-C}{\l}\|\dot{x}(t)\|;$$

(H$_3$) for $\ol{x}\in\omega(x)$ and $t_k\To+\infty$ such that $x(t_k)\To\ol{x}$ as $k\To+\infty$, and for every $a\ge 0$ we have
$$ H(\ddot{x}(t_k)+\g\dot{x}(t_k)+x(t_k),a\g\dot{x}(t)+x(t_k),\dot{x}(t_k))\To H(\ol{x},\ol{x},0) \ \mbox{as} \ k\To+\infty.$$
\end{lemma}

\begin{proof} $(H_1)$. From \eqref{ederiv} we have that for almost every $t \in [0,+\infty)$
$$\frac{d}{dt}\left[(f+g)( \ddot{x}(t)+\g\dot{x}(t)+x(t))+\frac{1}{2\l}\|\ddot{x}(t)+c\g\dot{x}(t)\|^2-\frac{C}{2\l}\|\dot{x}(t)\|^2\right]\le
A\|\dot{x}(t)\|^2+B\|\ddot{x}(t)\|^2.$$
Taking into account that $A<0,\,B<0$, we obtain that for almost every $t \in [0,+\infty)$
\begin{align*}
\frac{d}{dt}\left(H(\ddot{x}(t)+\g\dot{x}(t)+x(t),\g(1-c)\dot{x}(t)+x(t),\dot{x}(t))\right) & =\\
\frac{d}{dt}\left[(f+g)( \ddot{x}(t)+\g\dot{x}(t)+x(t))+\frac{1}{2\l}\|\ddot{x}(t)+c\g\dot{x}(t)\|^2-\frac{C}{2\l}\|\dot{x}(t)\|^2\right] & \le0.
\end{align*}

 By Lemma \ref{fejer-cont1} it follows that  the limit $$\lim_{t\To+\infty}H(\ddot{x}(t)+\g\dot{x}(t)+x(t),\g(1-c)\dot{x}(t)+x(t),\dot{x}(t))\in \R$$
exists.

$(H_2)$. From \eqref{inc} we have that for every $t\in[0,+\infty)$ it holds
$$-\frac1\l\ddot{x}(t)-\frac\g\l\dot{x}(t)-\n g(x(t))\in\p f(\ddot{x}(t)+\g\dot{x}(t)+x(t)),$$
hence
$$-\frac1\l\ddot{x}(t)-\frac\g\l\dot{x}(t)-\n g(x(t))+\n g(\ddot{x}(t)+\g\dot{x}(t)+x(t))\in\p (f+g)(\ddot{x}(t)+\g\dot{x}(t)+x(t)).$$
Since for every $(u,v,w) \in \R^n \times \R^n \times \R^n$
$$\p H(u,v,w)=\left(\p (f+g)(u)+\frac1\l(u-v)\right)\times\left\{-\frac1\l(u-v)\right\}\times\left\{-\frac{C}{\l}w\right\},$$
we get
$$\p H(\ddot{x}(t)+\g\dot{x}(t)+x(t),\g a\dot{x}(t)+x(t),\dot{x}(t))=$$
$$\left(\p (f+g)(\ddot{x}(t)+\g\dot{x}(t)+x(t))+\frac1\l(\ddot{x}(t)+(1-a)\g\dot{x}(t))\right)\times\left\{-\frac1\l(\ddot{x}(t)+(1-a)\g\dot{x}(t))\right\}\times\left\{-\frac{C}{\l}\dot{x}(t)\right\},$$
consequently,
\begin{align*}
w(t)=\left(-\n g(x(t))+\n g(\ddot{x}(t)+\g\dot{x}(t)+x(t))-\frac1\l a\g\dot{x}(t)),-\frac1\l(\ddot{x}(t)+(1-a)\g\dot{x}(t)),-\frac{C}{\l}\dot{x}(t)\right) & \in \\
\p H(\ddot{x}(t)+\g\dot{x}(t)+x(t),\g a\dot{x}(t)+x(t),\dot{x}(t)) &
\end{align*}
for every $t\in[0,+\infty)$.

From the $\b-$Lipschitz continuity of $\n g$ we get for every $t \in [0,+\infty)$
\begin{align*}
\|w(t)\|\le\left(\b+\frac1\l\right)\|\ddot{x}(t)+\g\dot{x}(t)\|+2\frac{a\g}{\l}\|\dot{x}(t)\|-\frac{C}{\l}\|\dot{x}(t)\| & \le\\
\left(\b+\frac1\l\right)\|\ddot{x}(t)\|+\frac{\b\l\g+(2a+1)\g-C}{\l}\|\dot{x}(t)\|.&
\end{align*}

$(H_3)$. Let $a\ge 0$, $\ol{x}\in\omega(x)$ and $t_k\To+\infty$ such that $x(t_k)\To\ol{x}$ as $k\To+\infty$. According to the proof of Lemma \ref{l9} it holds
$(f+g)(\ddot{x}(t_k)+\g\dot{x}(t_k)+x(t_k))\To (f+g)(\ol{x})$ as $k\To+\infty.$ Further, from Lemma \ref{l8} we have $\ddot{x}(t_k)\To 0$ and $\dot{x}(t_k)\To 0$ as $k\To+\infty.$
Hence,
$$H(\ddot{x}(t_k)+\g\dot{x}(t_k)+x(t_k),a\g\dot{x}(t_k)+x(t_k),\dot{x}(t_k))=(f+g)(\ddot{x}(t_k)+\g\dot{x}(t_k)+x(t_k))+$$
$$\frac{1}{2\l}\|\ddot{x}(t_k)+(1-a)\g\dot{x}(t_k)\|^2-\frac{C}{2\l}\|\dot{x}(t_k)\|^2\To(f+g)(\ol{x})= H(\ol{x},\ol{x},0) \ \mbox{as} \ k\To+\infty.$$
\end{proof}

\begin{lemma}\label{l11} Assume that $f+g$ is bounded from below and $\g,\l$ satisfy the set of conditions $(\rho)$, and let the constants $L, A, B$ and $C$ be defined as in Lemma \ref{l8}. For $u_0,v_0\in\R^n$, let $x\in C^2([0,+\infty),\R^n)$  be the unique global solution of \eqref{dysy}. Consider the function
$$H:\R^n\times \R^n\times \R^n\To\R\cup\{+\infty\},\,H(u,v,w)=(f+g)(u)+\frac{1}{2\l}\|u-v\|^2-\frac{C}{2\l}\|w\|^2.$$
Suppose that $x$ is bounded and let $a\ge 0.$ Then the following statements are true
\begin{itemize}
\item[(a)] $\omega(\ddot{x}+\g\dot{x}+x,a\g\dot{x}+x,\dot{x})\subseteq\crit(H)=\{(u,u,0)\in\R^n\times\R^n\times\R^n:u\in\crit(f+g)\};$
\item[(b)] $\ds\lim_{t\To+\infty}\dist((\ddot{x}(t)+\g\dot{x}(t)+x(t),a\g\dot{x}(t)+x(t),\dot{x}(t)),\omega(\ddot{x}+\g\dot{x}+x,a\g\dot{x}+x,\dot{x}))=0;$
\item[(c)] $H$ is finite and constant on $\omega(\ddot{x}+\g\dot{x}+x,a\g\dot{x}+x,\dot{x});$
\item[(d)] $\omega(\ddot{x}+\g\dot{x}+x,a\g\dot{x}+x,\dot{x})$ is nonempty, compact and connected.
\end{itemize}
\end{lemma}
\begin{proof} $(a)$ By definition, $$\omega(\ddot{x}+\g\dot{x}+x,a\g\dot{x}+x,\dot{x})=$$
$$\{(\ol{x},\ol{y},\ol z)\in(\R^n)^3:\exists t_k\to+\infty\mbox{ s. t. }(\ddot{x}(t_k)+\g\dot{x}(t_k)+x(t_k),a\g\dot{x}(t_k)+x(t_k),\dot{x}(t_k))\to(\ol{x},\ol{y},\ol z),\,k\to+\infty\}.$$
According to Lemma \ref{l8}, $\ddot{x}(t_k)\To 0,\,\dot{x}(t_k)\To0$ as $t_k\To+\infty$, hence
\begin{align*}
\omega(\ddot{x}+\g\dot{x}+x,a\g\dot{x}+x,\dot{x})=\omega(x,x,0) &= \\
\{(\ol{x},\ol{x},0)\in\R^n\times\R^n\times\R^n:\exists t_k\to+\infty\mbox{ such that }x(t_k)\To\ol{x},\,k\to+\infty\} &= \\
\{(\ol{x},\ol{x},0)\in\R^n\times\R^n\times\R^n:\ol{x}\in\omega(x)\}.
\end{align*}
According to Lemma \ref{l9},
$$\{(\ol{x},\ol{x},0)\in\R^n\times\R^n\times\R^n:\ol{x}\in\omega(x)\}\subseteq\{(\ol{x},\ol{x},0)\in\R^n\times\R^n\times\R^n:\ol{x}\in\crit(f+g)\}=\crit(H).$$

$(b)$ Obviously
$$0\le\lim_{t\To+\infty}\dist((\ddot{x}(t)+\g\dot{x}(t)+x(t),a\g\dot{x}(t)+x(t),\dot{x}(t)),\omega(\ddot{x}+\g\dot{x}+x,a\g\dot{x}+x,\dot{x}))\le$$
$$\lim_{t_k\To+\infty}\dist((\ddot{x}(t_k)+\g\dot{x}(t_k)+x(t_k),a\g\dot{x}(t_k)+x(t_k),\dot{x}(t_k)),\omega(\ddot{x}+\g\dot{x}+x,a\g\dot{x}+x,\dot{x}))=0.$$

$(c)$ According to Lemma \ref{l8}, $$\lim_{t\To+\infty}(f+g)( \ddot{x}(t)+\g\dot{x}(t)+x(t))=l\in\R.$$ Let $(\ol{x},\ol{x},0)\in \omega(\ddot{x}+\g\dot{x}+x,a\g\dot{x}+x,\dot{x}).$ Then there exists $t_k\To+\infty$ such that $(\ddot{x}(t_k)+\g\dot{x}(t_k)+x(t_k),a\g\dot{x}(t_k)+x(t_k),\dot{x}(t_k))\To (\ol{x},\ol{x},0)$ as $k\To+\infty.$ From Lemma \ref{l10}(H$_3$) one has
\begin{align*}
H(\ol{x},\ol{x},0)=\lim_{t_k\To+\infty}H(\ddot{x}(t_k)+\g\dot{x}(t_k)+x(t_k),a\g\dot{x}(t_k)+x(t_k),\dot{x}(t_k)) & =\\
\lim_{t_k\To+\infty}[(f+g)(\ddot{x}(t_k)+\g\dot{x}(t_k)+x(t_k))+\frac{1}{2\l}\|\ddot{x}(t_k)+(1-a)\g \dot{x}(t_k)\|^2-\frac{C}{2\l}\|\dot{x}(t_k)\|^2] & =l.
\end{align*}
Hence, $H$ takes on $\omega(\ddot{x}+\g\dot{x}+x,a\g\dot{x}+x,\dot{x})$ the constant value $l$.

Finally, $(d)$ is a classical result from \cite{haraux}. We also refer the reader  to the proof of Theorem 4.1 in 
\cite{alv-att-bolte-red}, where it is shown that the properties of $\omega(x)$ of being nonempty, compact and 
connected are generic for bounded trajectories fulfilling  $\lim_{t\rightarrow+\infty}{\dot x(t)}=0$ (see also \cite{b-sab-teb}
for a discrete version of this result).
\end{proof}

The convergence of the trajectory generated by the dynamical system \eqref{dysy} will be shown in the framework of functions satisfying 
the {\it Kurdyka-\L{}ojasiewicz property}.  For $\eta\in(0,+\infty]$, we denote by $\Theta_{\eta}$ 
the class of concave and continuous functions $\varphi:[0,\eta)\rightarrow [0,+\infty)$ such that $\varphi(0)=0$, $\varphi$ is 
continuously differentiable on $(0,\eta)$, continuous at $0$ and $\varphi'(s)>0$ for all
$s\in(0, \eta)$. In the following definition (see \cite{att-b-red-soub2010, b-sab-teb}) we use the {\it distance function} 
to a set, defined for $A\subseteq\R^n$ as $\dist(x,A)=\inf_{y\in A}\|x-y\|$ for all $x\in\R^n$.

\begin{definition}\label{KL-property} \rm({\it Kurdyka-\L{}ojasiewicz property}) Let $f:\R^n\rightarrow \R \cup \{+\infty\}$ be a proper 
and lower semicontinuous function. We say that $f$ satisfies the {\it Kurdyka-\L{}ojasiewicz (KL) property} at 
$\ol x\in \dom\partial f=\{x\in\R^n:\partial f(x)\neq\emptyset\}$
if there exist $\eta \in(0,+\infty]$, a neighborhood $U$ of $\ol x$ and a function $\varphi\in \Theta_{\eta}$ such that for all $x$ in the
intersection
$$U\cap \{x\in\R^n: f(\ol x)<f(x)<f(\ol x)+\eta\}$$ the following inequality holds
$$\varphi'(f(x)-f(\ol x))\dist(0,\partial f(x))\geq 1.$$
If $f$ satisfies the KL property at each point in $\dom\partial f$, then $f$ is called a {\it KL function}.
\end{definition}

The origins of this notion go back to the pioneering work of \L{}ojasiewicz \cite{lojasiewicz1963}, where it is proved that for a 
real-analytic function $f:\R^n\rightarrow\R$ and a critical point $\ol x\in\R^n$ (that is $\nabla f(\ol x)=0$), there exists $\theta\in[1/2,1)$ such that the function
$|f-f(\ol x)|^{\theta}\|\nabla f\|^{-1}$ is bounded around $\ol x$. This corresponds to the situation when 
$\varphi(s)=C(1-\theta)^{-1}s^{1-\theta}$. The result of \L{}ojasiewicz allows the interpretation of the KL property as a re-parametrization of the function values in order to avoid flatness around the
critical points. Kurdyka \cite{kurdyka1998} extended this property to differentiable functions definable in an o-minimal structure.
Further extensions to the nonsmooth setting can be found in \cite{b-d-l2006, att-b-red-soub2010, b-d-l-s2007, b-d-l-m2010}.

One of the remarkable properties of the KL functions is their ubiquity in applications, according to \cite{b-sab-teb}. To the class of KL functions belong semi-algebraic, real sub-analytic, semiconvex, uniformly convex and
convex functions satisfying a growth condition. We refer the reader to
\cite{b-d-l2006, att-b-red-soub2010, b-d-l-m2010, b-sab-teb, b-d-l-s2007, att-b-sv2013, attouch-bolte2009} and the references therein  for more details regarding all the classes mentioned above and illustrating examples.

An important role in our convergence analysis will be played by the following uniformized KL property given in \cite[Lemma 6]{b-sab-teb}.

\begin{lemma}\label{unif-KL-property} Let $\Omega\subseteq \R^n$ be a compact set and let $f:\R^n\rightarrow \R \cup \{+\infty\}$ be a proper
and lower semicontinuous function. Assume that $f$ is constant on $\Omega$ and $f$ satisfies the KL property at each point of $\Omega$.
Then there exist $\varepsilon,\eta >0$ and $\varphi\in \Theta_{\eta}$ such that for all $\ol x\in\Omega$ and for all $x$ in the intersection
\begin{equation}\label{int} \{x\in\R^n: \dist(x,\Omega)<\varepsilon\}\cap \{x\in\R^n: f(\ol x)<f(x)<f(\ol x)+\eta\}\end{equation}
the following inequality holds \begin{equation}\label{KL-ineq}\varphi'(f(x)-f(\ol x))\dist(0,\partial f(x))\geq 1.\end{equation}
\end{lemma}

We state the first main result of the paper.

\begin{theorem}\label{convergence}  Assume that $f+g$ is bounded from below and $\g,\l$ satisfy the set of conditions $(\rho)$, and let the constants $L, A, B$ and $C$ be defined as in Lemma \ref{l8}. For $u_0,v_0\in\R^n$, let $x\in C^2([0,+\infty),\R^n)$  be the unique global solution of \eqref{dysy}. Consider the function
$$H:\R^n\times \R^n\times \R^n\To\R\cup\{+\infty\},\,H(u,v,w)=(f+g)(u)+\frac{1}{2\l}\|u-v\|^2-\frac{C}{2\l}\|w\|^2.$$
Suppose that $x$ is bounded and $H$ is a KL function. Then the following statements are true
\begin{itemize}
\item[(a)] $\dot{x}\in L^1([0,+\infty),\R^n);$
\item[(b)] $\ddot{x}\in L^1([0,+\infty),\R^n);$
\item[(c)] there exists $\ol{x}\in\crit(f+g)$ such that $\lim_{t\To+\infty}x(t)=\ol{x}.$
\end{itemize}
\end{theorem}
\begin{proof} 
Let be $c:=\frac{L^2}{L^2+1}$. Consider an arbitrary $(\ol{x},\ol{x},0)\in \omega(\ddot{x}+\g\dot{x}+x,(1-c)\g\dot{x}+x,\dot{x}).$ Then one has
$$\lim_{t\To+\infty}\left(H(\ddot{x}(t)+\g\dot{x}(t)+x(t),\g(1-c)\dot{x}(t)+x(t),\dot{x}(t))\right)= H(\ol{x},\ol{x},0).$$

{\bf Case I.}
There exists $\ol{t}\ge0$ such that
$$H(\ddot{x}(\ol{t})+\g\dot{x}(\ol{t})+x(\ol{t}),\g(1-c)\dot{x}(\ol{t})+x(\ol{t}),\dot{x}(\ol t))= H(\ol{x},\ol{x},0).$$
We have for almost every $t\in[0,+\infty)$ that
$$\frac{d}{dt}\left[H(\ddot{x}(t)+\g\dot{x}(t)+x(t),\g(1-c)\dot{x}(t)+x(t),\dot{x}(t))\right]\le
A\|\dot{x}(t)\|^2+B\|\ddot{x}(t)\|^2\le 0.$$
Hence, for every $t\ge \ol{t}$ it holds
$$H(\ddot{x}(t)+\g\dot{x}(t)+x(t),\g(1-c)\dot{x}(t)+x(t),\dot{x}(t))\le H(\ol{x},\ol{x},0).$$
On the other hand
$$H(\ddot{x}(t)+\g\dot{x}(t)+x(t),\g(1-c)\dot{x}(t)+x(t),\dot{x}(t))\ge$$
$$\lim_{t\To+\infty}\left(H(\ddot{x}(t)+\g\dot{x}(t)+x(t),\g(1-c)\dot{x}(t)+x(t),\dot{x}(t))\right)= H(\ol{x},\ol{x},0),$$
hence
$$H(\ddot{x}(t)+\g\dot{x}(t)+x(t),\g(1-c)\dot{x}(t)+x(t),\dot{x}(t))= H(\ol{x},\ol{x},0)$$
for every $t\ge \ol{t}$.

Consequently,
$$\frac{d}{dt}\left[H(\ddot{x}(t)+\g\dot{x}(t)+x(t),\g(1-c)\dot{x}(t)+x(t),\dot{x}(t)\right]=0$$
for every $t\ge \ol{t}$, which means that
$$0\le A\|\dot{x}(t)\|^2+B\|\ddot{x}(t)\|^2\le 0$$
for every $t\ge \ol{t}.$

But $A<0$ and $B<0$, hence $\dot{x}(t)=0$ and $\ddot{x}(t)=0$ on $[\ol{t},+\infty).$ This leads to $\dot{x},\ddot{x}\in L^1([0,+\infty),\R^n)$ and to the fact hat $x(t)=\ol{x}$ is constant on $[\ol{t},+\infty).$

{\bf Case II.} For every $t\ge 0$
$$H(\ddot{x}(t)+\g\dot{x}(t)+x(t),\g(1-c)\dot{x}(t)+x(t),\dot{x}(t))> H(\ol{x},\ol{x},0).$$
Let $\Omega=\omega(\ddot{x}+\g\dot{x}+x,(1-c)\g\dot{x}+x,\dot{x}).$ According to Lemma \ref{l11}, $H$ is constant and finite on $\Omega$ and $\Omega$ is nonempty, compact and connected. Since $H$ is a KL function, by Lemma \ref{unif-KL-property}, there exist $\varepsilon,\eta >0$ and a concave function $\varphi\in \Theta_{\eta}$ such that for every $(\ol x,\ol x,0)\in\Omega$ and every
\begin{align}\label{intersect}(x,y,z)\in\{(u,v,w)\in\R^n\times \R^n\times\R^n: \dist((u,v,w),\Omega)<\varepsilon\}  & \cap \nonumber \\
\{(u,v,w)\in\R^n\times\R^n\times\R^n: H(\ol x,\ol x,0)<H(u,v,w)<H(\ol x,\ol x,0)+\eta\} &
\end{align}
the following inequality holds \begin{equation}\label{kline}\varphi'(H(x,y,z)-H(\ol x,\ol x,0))\dist((0,0,0),\partial H(x,y,z))\geq 1.\end{equation}

Since $$\lim_{t\To+\infty}\left(H(\ddot{x}(t)+\g\dot{x}(t)+x(t),\g(1-c)\dot{x}(t)+x(t),\dot{x}(t))\right)= H(\ol{x},\ol{x},0)$$
and
$$H(\ddot{x}(t)+\g\dot{x}(t)+x(t),\g(1-c)\dot{x}(t)+x(t),\dot{x}(t))> H(\ol{x},\ol{x},0),$$
there exists $t_1>0$ such that
$$H(\ddot{x}(t)+\g\dot{x}(t)+x(t),\g(1-c)\dot{x}(t)+x(t),\dot{x}(t))< H(\ol{x},\ol{x},0)+\eta \,\forall t\ge t_1.$$

Since $\lim_{t\To+\infty}\dist((\ddot{x}(t)+\g\dot{x}(t)+x(t),\g(1-c)\dot{x}(t)+x(t),\dot{x}(t)),\Omega)=0$, there exists $t_2\ge 0$ such that
$$\dist((\ddot{x}(t)+\g\dot{x}(t)+x(t),\g(1-c)\dot{x}(t)+x(t),\dot{x}(t)),\Omega)<\e,\,\forall t\ge t_2.$$
Hence, for every $t\ge T=\max(t_1,t_2)$ we have
\begin{align*}
\varphi'(H(\ddot{x}(t)+\g\dot{x}(t)+x(t),\g(1-c)\dot{x}(t)+x(t),\dot{x}(t))-H(\ol{x},\ol{x},0))\cdot &\\
\dist((0,0,0),\partial H(\ddot{x}(t)+\g\dot{x}(t)+x(t),\g(1-c)\dot{x}(t)+x(t),\dot{x}(t))) & \ge 1.
\end{align*}

On the other hand, for every $t \in [T, +\infty)$,
$$\dist((0,0,0),\partial H(\ddot{x}(t)+\g\dot{x}(t)+x(t),\g(1-c)\dot{x}(t)+x(t),\dot{x}(t)))\le\|w(t)\|,$$
where $$w(t)=\left(-\n g(x(t))+\n g(\ddot{x}(t)+\g\dot{x}(t)+x(t))-\frac1\l (1-c)\g\dot{x}(t)),-\frac1\l(\ddot{x}(t)+c\g\dot{x}(t)),-\frac{C}{\l}\dot{x}(t)\right)$$ since, according to Lemma \ref{l10} $(H_2)$,
$$w(t)\in\p H(\ddot{x}(t)+\g\dot{x}(t)+x(t),\g (1-c)\dot{x}(t)+x(t),\dot{x}(t)).$$
Further,
$$\|w(t)\|\le\left(\b+\frac1\l\right)\|\ddot{x}(t)\|+\frac{\b\l\g+(3-2c)\g-C}{\l}\|\dot{x}(t)\|$$ which leads to
$$\varphi'(H(\ddot{x}(t)+\g\dot{x}(t)+x(t),\g(1-c)\dot{x}(t)+x(t),\dot{x}(t))-H(\ol{x},\ol{x},0))\left( s\|\ddot{x}(t)\|+p\|\dot{x}(t)\|\right)\ge 1 \ \forall t \in [T, +\infty),$$
where
$s:=\b+\frac1\l>0$ and $p:=\frac{\b\l\g+(3-2c)\g-C}{\l}>0.$

We have
\begin{align*}
\frac{d}{dt}\varphi(H(\ddot{x}(t)+\g\dot{x}(t)+x(t),\g(1-c)\dot{x}(t)+x(t),\dot{x}(t))-H(\ol{x},\ol{x},0)) &=\\
\varphi'(H(\ddot{x}(t)+\g\dot{x}(t)+x(t),\g(1-c)\dot{x}(t)+x(t),\dot{x}(t))-H(\ol{x},\ol{x},0))\cdot\\
\frac{d}{dt}H(\ddot{x}(t)+\g\dot{x}(t)+x(t),\g(1-c)\dot{x}(t)+x(t),\dot{x}(t))
\end{align*}
and since
$$\frac{d}{dt}H(\ddot{x}(t)+\g\dot{x}(t)+x(t),\g(1-c)\dot{x}(t)+x(t),\dot{x}(t))\le A\|\dot{x}(t)\|^2+B\|\ddot{x}(t)\|^2\le0$$
and
$$\varphi'(H(\ddot{x}(t)+\g\dot{x}(t)+x(t),\g(1-c)\dot{x}(t)+x(t),\dot{x}(t))-H(\ol{x},\ol{x},0))\ge\frac{1}{ s\|\ddot{x}(t)\|+p\|\dot{x}(t)\|}$$
we get for every $t \in [T, +\infty)$
\begin{equation}\label{e15}
\frac{d}{dt}\varphi(H(\ddot{x}(t)+\g\dot{x}(t)+x(t),\g(1-c)\dot{x}(t)+x(t),\dot{x}(t))-H(\ol{x},\ol{x},0))\le\frac{A\|\dot{x}(t)\|^2+B\|\ddot{x}(t)\|^2}
{s\|\ddot{x}(t)\|+p\|\dot{x}(t)\|}\le0.
\end{equation}
Since $\varphi$ is bounded from below, similarly as in the proof of Lemma \ref{l8}, we obtain that
$$\frac{\|\dot{x}(\cdot)\|^2}{s\|\ddot{x}(\cdot)\|+p\|\dot{x}(\cdot)\|},\frac{\|\ddot{x}(\cdot)\|^2}{s\|\ddot{x}(\cdot)\|+p\|\dot{x}(\cdot)\|}\in L^1([0,+\infty),\R).$$
By using the arithmetical-geometrical mean inequality we have $$\ds\sqrt{\frac{\|\dot{x}(\cdot)\|^2}{s\|\ddot{x}(\cdot)\|+p\|\dot{x}(\cdot)\|}\cdot\frac{\|\ddot{x}(\cdot)\|^2}{s\|\ddot{x}(\cdot)\|+p\|\dot{x}(\cdot)\|}}= \frac{\|\dot{x}(\cdot)\|\|\ddot{x}(\cdot)\|}{s\|\ddot{x}(\cdot)\|+p\|\dot{x}(\cdot)\|}\in L^1([0,+\infty),\R).$$
Hence,
$$\|\dot{x}(\cdot)\|+\|\ddot{x}(\cdot)\|= p \frac{\|\dot{x}(\cdot)\|^2}{s\|\ddot{x}(\cdot)\|+p\|\dot{x}(\cdot)\|}+s \frac{\|\ddot{x}(\cdot)\|^2}{s\|\ddot{x}(\cdot)\|+p\|\dot{x}(\cdot)\|}+(s+p)\frac{\|\dot{x}(\cdot)\|\|\ddot{x}(\cdot)\|}{s\|\ddot{x}(\cdot)\|+p\|\dot{x}(\cdot)\|}\in L^1([0,+\infty),\R).$$
This shows that $\dot{x},\ddot{x}\in L^1([0,+\infty),\R^n)$, hence, according to Lemma \ref{fejer-cont1}, there exists
$\lim_{t\To+\infty}x(t) = \ol x$.
\end{proof}

\begin{remark} Similar regularizations of the objective function as the one considered in this section
have been used in \cite{bc-forder-kl} for studying first order dynamical systems, but also in \cite{bcl, ipiano}, in the investigation of non-relaxed forward-backward methods involving inertial and memory effects in the nonconvex setting.
\end{remark}

\begin{remark} Since the class of semi-algebraic functions is closed under addition (see for example \cite{b-sab-teb}) and 
$(u,v) \mapsto \alpha\|u-v\|^2$ and $w\mapsto \alpha'\|w\|^2$ are semi-algebraic for $\alpha,\alpha'>0$, the conclusion of the 
previous theorem holds if the condition $H$ is a KL function is replaced by the assumption that $f+g$ is semi-algebraic.  
\end{remark}

\begin{remark}\label{cond-x-bound} 
Assume that $\gamma, \lambda>0$ fulfill the set of conditions $(\rho)$ and that $f+g$ is coercive, 
that is $$\lim_{\|u\|\rightarrow+\infty}(f+g)(u)=+\infty.$$ 
For $u_0,v_0\in\R^n$, let  $x \in C^2([0,+\infty), \R^n)$ be the unique global solution of 
\eqref{dysy}. Then $x$ is bounded.  

Indeed, notice that $f+g$ is bounded from below, being a proper, lower semicontinuous and coercive function (see for example \cite{rock-wets}). From \eqref{integ} it follows that $\ddot x(T)+\gamma \dot x(T)+x(T)$ is contained for every $T\geq 0$ in a lower level set 
of $f+g$, which is a bounded set due to the coercivity assumption. Combining this fact with Lemma \ref{l8} one can easily derive that $x$  is bounded. 
\end{remark}

\section{Convergence rates}\label{sec5}

In the context of optimization problems involving KL functions, it is known (see \cite{lojasiewicz1963, b-d-l2006, attouch-bolte2009}) 
that convergence rates of the trajectory can be formulated in terms of the so-called \L{}ojasiewicz exponent.

\begin{definition}\label{Lexpo}  Let $f:\R^n\To \R \cup \{+\infty\}$ be a proper and lower semicontinuous function. The function $f$ is said to fulfill the {\L}ojasiewicz property, if for every $\ol x\in\crit{f}$ there exist $K,\e>0$ and $\t\in(0,1)$ such that
$$|f(x)-f(\ol x)|^\t\le K\|x^*\|\mbox{ for every }x\mbox{ fulfilling }\|x-\ol x\|<\e\mbox{ and every }x^*\in\p f(x).$$

The number $\t$  is called the {\L}ojasiewicz exponent of $f$ at the critical  point $\ol x.$
\end{definition}

In the following theorem we obtain convergence rates for both the trajectory generated \eqref{dysy} and its velocity (see, also, \cite{b-d-l2006, attouch-bolte2009}).

\begin{theorem}\label{th-conv-rate}  Assume that $f+g$ is bounded from below and $\g,\l$ satisfy the set of conditions $(\rho)$, and let the constants $L, A, B$ and $C$ be defined as in Lemma \ref{l8}. For $u_0,v_0\in\R^n$, let $x\in C^2([0,+\infty),\R^n)$  be the unique global solution of \eqref{dysy}. Consider the function
$$H:\R^n\times \R^n\times \R^n\To\R\cup\{+\infty\},\,H(u,v,w)=(f+g)(u)+\frac{1}{2\l}\|u-v\|^2-\frac{C}{2\l}\|w\|^2.$$
Suppose that $x$ is bounded and let $\ol x\in\crit(f+g)$ be such that $\lim_{t\To+\infty}x(t)=\ol x$ and $H$ fulfills the {\L}ojasiewicz property at $(\ol x,\ol x,0)\in\crit H$ with {\L}ojasiewicz exponent  $\t$.

Then, there exist $a_1,a_2,a_3,a_4>0$ and $t_0>0$ such that for every $t \in [t_0, +\infty)$ the following statements are true
\begin{itemize}
\item[(a)] if $\t\in(0,\frac12),$ then $x$ converges in finite time;
\item[(b)] if $\t=\frac12,$ then $\|x(t)-\ol x\|\le a_1 e^{-a_2t}$ and $\|\dot{x}(t)\|\le a_1 e^{-a_2t}$ ;
\item[(c)] if $\t\in(\frac12,1),$ then $\|x(t)-\ol x\|\le (a_3t+a_4)^{-\frac{1-\t}{2\t-1}}$ and $\|\dot{x}(t)\|\le (a_3t+a_4)^{-\frac{1-\t}{2\t-1}}.$
\end{itemize}
\end{theorem}
\begin{proof} Let be $s:=\b+\frac1\l>0$ and $p:=\frac{\b\l\g+(3-2c)\g-C}{\l}>0$, as defined in Lemma \ref{l10}. The function $g:[0,+\infty)\To\R,\,g(r)=\frac{A+Br^2}{p+(s+p)r+sr^2}$ attains at $r_0=\frac{(sA-pB)-\sqrt{(sA-pB)^2+(s+p)^2 AB}}{(s+p)B}>0$ its maximum. Hence, for $m:=\max\left(\frac{B}{s}, g(r_0)\right) < 0$, it holds 
$$A\|\dot{x}(t)\|^2+B\|\ddot{x}(t)\|^2 \le m(s\|\ddot{x}(t)\|+p\|\dot{x}(t)\|)(\|\dot{x}(t)\|+\|\ddot{x}(t)\|)$$ for every $t\in[0,+\infty).$

We define for every $t\in [0, +\infty)$
$$\s(t):=\int_t^{+\infty}(\|\dot{x}(s)\|+\|\ddot{x}(t)\|)ds.$$
Let $t\in [0, +\infty)$ be fixed. For $T\ge t$ we have
$$\|x(t)-\ol x\|=\left \|x(T)-\ol x-\int_t^{T}\dot{x}(s)ds \right\|\le\|x(T)-\ol x\|+\int_t^{T}\|\dot{x}(s)\|ds.$$
By taking the limit as  $T\To+\infty$ we obtain
\begin{equation}\label{e17}
\|x(t)-\ol x\|\le\int_t^{+\infty}\|\dot{x}(s)\|ds\le\s(t).
\end{equation}
Further, for $T\ge t$ we have
$$\|\dot{x}(t)\|= \left \|\dot{x}(T)-\int_t^{T}\ddot{x}(s)ds \right \|\le\|\dot{x}(T)\|+\int_t^{T}\|\ddot{x}(s)\|ds.$$
By taking the limit as $T\To+\infty$ we obtain
\begin{equation}\label{e18}
\|\dot{x}(t)\|\le\int_t^{+\infty}\|\ddot{x}(s)\|ds\le\s(t).
\end{equation}

We have seen in the proof of Theorem \ref{convergence} that, if there exists $\ol{t}\ge0$ such that
$$H(\ddot{x}(\ol{t})+\g\dot{x}(\ol{t})+x(\ol{t}),\g(1-c)\dot{x}(\ol{t})+x(\ol{t}),\dot{x}(\ol t))= H(\ol{x},\ol{x},0),$$
then $x$ is constant on $[\ol t,+\infty)$, hence the conclusion follows automatically.

On the other hand, if for every $t\ge 0$ one has
$$H(\ddot{x}(t)+\g\dot{x}(t)+x(t),\g(1-c)\dot{x}(t)+x(t),\dot{x}(t))> H(\ol{x},\ol{x},0),$$
then, according to the proof of Theorem \ref{convergence} and \eqref{e15}, there exists $t_0\ge 0$ such that for every $t\in [t_0, +\infty)$
$$K\frac{d}{dt}(H(\ddot{x}(t)+\g\dot{x}(t)+x(t),\g(1-c)\dot{x}(t)+x(t),\dot{x}(t))-H(\ol{x},\ol{x},0))^{1-\t}\le\frac{A\|\dot{x}(t)\|^2+B\|\ddot{x}(t)\|^2}
{s\|\ddot{x}(t)\|+p\|\dot{x}(t)\|},$$
and
$$\|(\ddot{x}(t)+\g\dot{x}(t)+x(t),(1-c)\g\dot{x}(t)+x(t),\dot{x}(t)) - (\ol x,\ol x,0)\|<\e.$$

Hence, for every $t\in [t_0, +\infty)$
\begin{equation}\label{e16}M(\|\dot{x}(t)\|+\|\ddot{x}(t)\|)+\frac{d}{dt}(H(\ddot{x}(t)+\g\dot{x}(t)+x(t),\g(1-c)\dot{x}(t)+x(t),\dot{x}(t))-H(\ol{x},\ol{x},0))^{1-\t}\le 0,
\end{equation}
and
$$\|(\ddot{x}(t)+\g\dot{x}(t)+x(t),(1-c)\g\dot{x}(t)+x(t),\dot{x}(t))- (\ol x,\ol x,0)\|<\e,$$
where $M:=-\frac{m}{K}>0.$
If we integrate \eqref{e16} on the interval $[t,T]$, where $T \geq  t\ge t_0$, we obtain
\begin{align*}
M\int_t^T(\|\dot{x}(s)\|+\|\ddot{x}(s)\|)ds+(H(\ddot{x}(T)+\g\dot{x}(T)+x(T),\g(1-c)\dot{x}(T)+x(T),\dot{x}(T))-H(\ol{x},\ol{x},0))^{1-\t} & \le\\
(H(\ddot{x}(t)+\g\dot{x}(t)+x(t),\g(1-c)\dot{x}(t)+x(t),\dot{x}(t))-H(\ol{x},\ol{x},0))^{1-\t},&
\end{align*}
hence
$$M\s(t)\le (H(\ddot{x}(t)+\g\dot{x}(t)+x(t),\g(1-c)\dot{x}(t)+x(t),\dot{x}(t))-H(\ol{x},\ol{x},0))^{1-\t} \ \quad \forall t\ge t_0.$$
Since $\t$ is the {\L}ojasiewicz exponent of $H$ at the   point $(\ol x,\ol x,0)\in\crit H,$ we have
$$|H(\ddot{x}(t)+\g\dot{x}(t)+x(t),\g(1-c)\dot{x}(t)+x(t),\dot{x}(t))-H(\ol{x},\ol{x},0)|^{\t}\le K\|x^*\|,$$
for every $t\in [t_0, +\infty)$ and every
$$x^*\in\p H(\ddot{x}(t)+\g\dot{x}(t)+x(t),\g(1-c)\dot{x}(t)+x(t),\dot{x}(t)).$$
According to Lemma \ref{l10}$(H_2)$, there exists some $\widetilde x^*\in\p H( \ddot{x}(t)+\g\dot{x}(t)+x(t),\g(1-c)\dot{x}(t)+x(t),\dot{x}(t))$ such that for almost every $t\in [t_0, +\infty)$
$$\|\widetilde x^*(t)\|\le s\|\ddot{x}(t)\|+p\|\dot{x}(t)\|\le N(\|\ddot{x}(t)\|+\|\dot{x}(t)\|),$$
where $N=\max(s,p).$
Hence,
$$M\s(t)\le(KN(\|\ddot{x}(t)\|+\|\dot{x}(t)\|))^{\frac{1-\t}{\t}}$$
for almost every $t\in [t_0, +\infty)$. But $\dot{\s}(t)=-\|\ddot{x}(t)\|-\|\dot{x}(t)\|$, consequently, there exists $\a>0$ such that for almost every $t\in [t_0, +\infty)$
\begin{equation}\label{e19}
\dot{\s}(t)\le-\a(\s(t))^{\frac{\t}{1-\t}}.
\end{equation}

If $\t=\frac12$, then $\dot{\s}(t)\le-\a(\s(t))$ for almost every $t\in [t_0, +\infty)$. By multiplying with $e^{\a t}$ and integrating on $[t_0,t]$, we get that there exist $a_1,a_2>0$ such that $$\s(t)\le a_1 e^{-a_2t} \ \forall t\in [t_0, +\infty),$$
hence, by \eqref{e17} and \eqref{e18}, we get  $$\|x(t)-\ol x\|\le a_1 e^{-a_2t}\mbox{ and }\|\dot{x}(t)\|\le a_1 e^{-a_2t} \ \forall t\in [t_0, +\infty),$$ which proves $(b)$.

Assume now that $0<\t<\frac12.$ By using \eqref{e19} we obtain
$$\frac{d}{dt}\left(\s(t)\right)^{\frac{1-2\t}{1-\t}}=\frac{1-2\t}{1-\t}\left(\s(t)\right)^{\frac{-\t}{1-\t}}\dot{\s}(t)\le-\a\frac{1-2\t}{1-\t},$$
for almost every $t\in [t_0, +\infty)$.

By integration we get
$$\left(\s(t)\right)^{\frac{1-2\t}{1-\t}}\le-\ol \a t+\ol \b \ \forall t\in [t_0, +\infty),$$
where $\ol \a>0.$ Hence, there exists $T\ge 0$ such that $\s(T)\le 0 \ \forall t \in [T, +\infty),$ which implies that $x$ is constant on $[T,+\infty).$

Assume now that $\frac12<\t<1.$
By using \eqref{e19} we obtain
$$\frac{d}{dt}\left(\s(t)\right)^{\frac{1-2\t}{1-\t}}=\frac{1-2\t}{1-\t}\left(\s(t)\right)^{\frac{-\t}{1-\t}}\dot{\s}(t)\ge\a\frac{2\t-1}{1-\t}$$
for almost every $t\in [t_0, +\infty)$.

By integration we get
$$\s(t)\le(a_3 t+a_4)^{-\frac{1-\t}{2\t-1}} \ \forall t\in [t_0, +\infty),$$
where $a_3,a_4>0.$

From \eqref{e17} and \eqref{e18} we have  $$\|x(t)-\ol x\|\le (a_3 t+a_4)^{-\frac{1-\t}{2\t-1}}\mbox{ and }\|\dot{x}(t)\|\le (a_3 t+a_4)^{-\frac{1-\t}{2\t-1}} \ \forall t\in [t_0, +\infty),$$ which proves $(c)$.
\end{proof}

\end{document}